\documentclass[english,a4paper]{amsart}
\usepackage{amssymb}
\usepackage{amsfonts}
\usepackage{amscd}
\usepackage{amsthm}
\usepackage{color}
\usepackage{enumerate}
\usepackage{todonotes}
\usepackage{enumitem}

\usepackage[T1]{fontenc}
\RequirePackage[colorlinks=true, breaklinks=true, urlcolor= black, linkcolor= black, citecolor= black, bookmarksopen=true,linktocpage=true,plainpages=false,pdfpagelabels,unicode]{hyperref}

\usepackage{cleveref}

\def\deg{{\rm deg}}

\def\11{{\mathbf 1}}
\def\AA{{\mathbb A}}

\def\FF{{\mathbb F}}

\def\PP{{\mathbb P}}
\def\QQ{{\mathbb Q}}

\def\ZZ{{\mathbb Z}}

\mathchardef\alphag="7C0B \mathchardef\betag="7C0C
\mathchardef\gammag="7C0D \mathchardef\deltag="7C0E
\mathchardef\varepsilong="7C22 \mathchardef\varphig="7C27
\mathchardef\psig="7C20 \mathchardef\zetag="7C10
\mathchardef\epsilong="7C0F \mathchardef\rhog="7C1A
\mathchardef\taug="7C1C \mathchardef\upsilong="7C1D
\mathchardef\iotag="7C13 \mathchardef\thetag="7C12
\mathchardef\pig="7C19 \mathchardef\sigmag="7C1B
\mathchardef\etag="7C11 \mathchardef\omegag="7C21
\mathchardef\kappag="7C14 \mathchardef\lambdag="7C15
\mathchardef\mug="7C16 \mathchardef\xig="7C18
\mathchardef\chig="7C1F \mathchardef\nug="7C17
\mathchardef\varthetag="7C23 \mathchardef\varpig="7C24
\mathchardef\varrhog="7C25 \mathchardef\varsigmag="7C26
\mathchardef\Omegag="7C0A \mathchardef\Thetag="7C02
\mathchardef\Sigmag="7C06 \mathchardef\Deltag="7C01
\mathchardef\Phig="7C08 \mathchardef\Gammag="7C00
\mathchardef\Psig="7C09 \mathchardef\Lambdag="7C03
\mathchardef\Xig="7C04 \mathchardef\Pig="7C05
\mathchardef\Upsilong="7C07

\newcounter{theoremcntr}[subsection]
\renewcommand*{\thetheoremcntr}{%
  \ifnum\value{subsection}=0 %
    \thesection
  \else
    \thesubsection
  \fi
  .\arabic{theoremcntr}%
}
\numberwithin{figure}{section}
\theoremstyle{plain}
\newtheorem{thm}{\protect\theoremname}[section]
\theoremstyle{plain}
\newtheorem*{thm*}{\protect\theoremname}
\theoremstyle{plain} 
\theoremstyle{remark}
\newtheorem{remark}[thm]{\protect\remarkname}
\theoremstyle{definition}

\theoremstyle{plain}

\theoremstyle{plain}
\newtheorem{cor}[thm]{\protect\corollaryname}
\theoremstyle{plain}
\newtheorem{prop}[thm]{\protect\propositionname}
\theoremstyle{plain}
\newtheorem*{prop*}{\protect\propositionname}
\theoremstyle{plain}
\newtheorem{lem}[thm]{\protect\lemmaname}
\theoremstyle{plain}
\newtheorem*{lem*}{\protect\lemmaname}
\theoremstyle{plain}
\newtheorem{conj}[thm]{\protect\conjecturename}
\theoremstyle{plain}

\theoremstyle{plain}

\providecommand{\corollaryname}{Corollary}
\providecommand{\definitionname}{Definition}
\providecommand{\notationname}{Notation}
\providecommand{\lemmaname}{Lemma}
\providecommand{\propositionname}{Proposition}
\providecommand{\remarkname}{Remark}
\providecommand{\theoremname}{Theorem}
\providecommand{\conjecturename}{Conjecture}
\providecommand{\examplename}{Example}

\Crefname{prop}{Proposition}{Propositions}
\Crefname{thm}{Theorem}{Theorems}
\Crefname{lem}{Lemma}{Lemmas}
\Crefname{cor}{Corollary}{Corollaries}
\Crefname{maintheorem}{Theorem}{Theorems}
\def\boxit#1#2{\setbox1=\hbox{\kern#1{#2}\kern#1}%
\dimen1=\ht1 \advance\dimen1 by #1 \dimen2=\dp1 \advance\dimen2 by
#1
\setbox1=\hbox{\vrule height\dimen1 depth\dimen2\box1\vrule}%
\setbox1=\vbox{\hrule\box1\hrule}%
\advance\dimen1 by .4pt \ht1=\dimen1 \advance\dimen2 by .4pt
\dp1=\dimen2 \box1\relax}

\mathchardef\alphag="7C0B \mathchardef\betag="7C0C
\mathchardef\gammag="7C0D \mathchardef\deltag="7C0E
\mathchardef\varepsilong="7C22 \mathchardef\varphig="7C27
\mathchardef\psig="7C20 \mathchardef\zetag="7C10
\mathchardef\epsilong="7C0F \mathchardef\rhog="7C1A
\mathchardef\taug="7C1C \mathchardef\upsilong="7C1D
\mathchardef\iotag="7C13 \mathchardef\thetag="7C12
\mathchardef\pig="7C19 \mathchardef\sigmag="7C1B
\mathchardef\etag="7C11 \mathchardef\omegag="7C21
\mathchardef\kappag="7C14 \mathchardef\lambdag="7C15
\mathchardef\mug="7C16 \mathchardef\xig="7C18
\mathchardef\chig="7C1F \mathchardef\nug="7C17
\mathchardef\varthetag="7C23 \mathchardef\varpig="7C24
\mathchardef\varrhog="7C25 \mathchardef\varsigmag="7C26
\mathchardef\Omegag="7C0A \mathchardef\Thetag="7C02
\mathchardef\Sigmag="7C06 \mathchardef\Deltag="7C01
\mathchardef\Phig="7C08 \mathchardef\Gammag="7C00
\mathchardef\Psig="7C09 \mathchardef\Lambdag="7C03
\mathchardef\Xig="7C04 \mathchardef\Pig="7C05
\mathchardef\Upsilong="7C07

\definecolor{orange}{rgb}{1,0.5,0}

\author[Dehennin]{Luca Dehennin}
\address{KU Leuven, Department of Mathematics, B-3001 Leu\-ven, Bel\-gium}
\email{luca.dehennin@kuleuven.be}
\urladdr{}


\title[Improved projective Schwartz-Zippel and dimension growth bounds]{Improved, sublinear projective Schwartz-Zippel and (sub)quadratic dimension growth bounds in arbitrary codimension}

\makeatother

\begin{document}

\begin{abstract}
    We introduce a sublinear sharpened version of a classical projective variant of the Schwartz-Zippel bounds. Here, the sublinearity is in the degree, whereas the classical Schwartz-Zippel bounds are linear in the degree. We prove such sublinear bounds in several cases, including for unions of linear varieties. These bounds lead to a subquadratic dependence on the degree for projective dimension growth bounds in low dimensions. This improves on the quadratic dependence obtained by Binyamini, Cluckers and Kato in \cite{BinCluKat} and thus brings the dimension growth upper bounds and lower bounds for the worst case closer together, a question asked by Cluckers and Glazer in \cite{CGlaz}. We also address another question from \cite{CGlaz}, to treat the case of general codimension. To this end, we introduce a natural projection argument with pull-backs. We also give our main results for any global field instead of just $\QQ$.
\end{abstract}

\maketitle
	
\section{Introduction}
\renewcommand{\thethm}{\arabic{thm}}

Recently, Binyamini, Cluckers and Kato obtained quadratic dependence on the degree in dimension growth bounds for hypersurfaces. Their approach relies on the $p$-adic determinant method and the well-known Schwartz-Zippel bound. We sharpen the dependence on the degree using an improved, sublinear version of the Schwartz-Zippel bound, and we extend the results of the determinant method obtained in \cite{BinCluKat} to arbitrary codimension using a natural projection argument. This leads to subquadratic dimension growth in $\PP^3_\QQ$, $\PP^4_\QQ$ and $\PP^5_\QQ$ and to quadratic dimension growth for varieties of any codimension in $\PP^n_\QQ$. We also apply this projection argument to bound integer points on affine varieties in arbitrary codimension, and we generalise our main results to arbitrary global fields.

We use the following notation. For $X\subset \PP_\QQ^n$ and $B \ge 1$ we denote by $X(\QQ, B)$ the set of rational points of $X$ of height $\le B$, that is, the set of points $x \in X$ that can be written in homogeneous integer coordinates $(x_0, \dots, x_n)$ with $|x_0|,\dots,|x_n| \le B$. We study uniform upper bounds for $\#X(\QQ, B)$ and their dependence on the degree of $X$. By a variety $X$ we mean a reduced subscheme of $\AA^n_\QQ$ or of $\PP^n_\QQ$. Throughout this text, the notation $A \ll B$ denotes $A \le cB$ for some constant $c = c(K, n, m)$ depending only on the dimensions $n$ and $m$ and the field $K$.

\subsection{Schwartz-Zippel}

    Recall the well-known variant of the Schwartz-Zippel bound for affine varieties\footnote{\label{foot:trivial}In many related literature, the cases of Schwartz-Zippel discussed here have been named the "Trivial bounds", for example in \cite{Brow-Heath-Salb}.}. For hypersurfaces, this is a special case of the Schwartz-Zippel lemma from theoretical computer science; see \cite{Schwartz} and \cite{Zippel}, which applies to general sets instead of only to integer points of bounded height.
    
    \begin{lem*}[Affine variant of Schwartz-Zippel]\label{ASZ}
    Let $X$ be an affine variety of degree $d$ and pure dimension $m$ in $\AA_\QQ^n$. Then for all $B\ge 1$ one has \[\#X(\ZZ, B)\le d(2B+1)^{m}.\]
    \end{lem*}
    \begin{proof}
        This statement follows from the proof of \cite[Theorem 1]{BrowningHeathBrown-Crelle}.
    \end{proof}
    
    The affine Schwartz-Zippel bound gives the best possible dependence on $B$ for arbitrary varieties. It is also optimal in the degree as it is reached by the union of $d$ parallel linear varieties. Considering the affine cone over a projective variety produces the following projective analogue.

    \begin{lem*}[Projective variant of Schwartz-Zippel]\label{SZ}
    Let $X$ be a projective variety of degree $d$ and pure dimension $m$ in $\PP_\QQ^n$. Then for all $B\ge 1$ one has \[\#X(\QQ, B)\le cdB^{m+1}\] for some constant $c = c(m)$ depending only on $m$.
    \end{lem*}

    The projective Schwartz-Zippel bound also gives the best possible dependence on the bound $B$ for arbitrary varieties. However, the dependence on the degree is not necessarily optimal, as this bound cannot be reached by the same example as in the affine case since $d$ parallel linear varieties do not arise as the affine cone over a projective variety. We show that there is, in fact, no union of $d$ linear varieties that reaches this bound. We establish the following stronger bound for the number of rational points on a union of linear varieties in $\PP^n$. 

    \begin{thm}[Sublinear projective variant of Schwartz-Zippel for a union of linear varieties]\label{SLSZLin}
        Let $X$ be a union of $d$ linear varieties, each of dimension $m$ in $\PP_\QQ^n$. Then for all $B\ge 1$ one has \[\#X(\QQ, B)\le cd^{\frac{n}{n+1}}B^{m+1}\] for some constant $c = c(n)$ depending only on $n$.
    \end{thm}

    We also obtain a generalisation of Theorem \ref{SLSZLin} to global fields; see Theorem \ref{SLSZLin:global}.
    
    If we assume that, as in the affine case, the optimal bound is reached by a union of $d$ linear varieties, this constitutes the following improvement to the classical Schwartz-Zippel bound for projective varieties.

    \begin{conj}[Sublinear projective variant of Schwartz-Zippel]\label{SLSZ}
        Let $X$ be a projective variety of degree $d$ and pure dimension $m$ in $\PP_\QQ^n$. Then for all $B\ge 1$ one has \[\#X(\QQ, B)\le cd^{\frac{n}{n+1}}B^{m+1}\] for some constant $c = c(n)$ depending only on $n$.
    \end{conj}

    The dependence on the degree in both Theorem \ref{SLSZLin} and Conjecture \ref{SLSZ} is optimal. We show this by proving that the bounds are reached by a concrete union of linear varieties given by Proposition \ref{planecover}. This is discussed in more detail in Remark \ref{rem:optimal}. 
    
    In this paper, we establish Conjecture \ref{SLSZ} for unions of linear varieties by proving Theorem \ref{SLSZLin} using a result of Schmidt \cite{Schmidt-lattices}. For arbitrary varieties, we establish Conjecture \ref{SLSZ} for $n=2$ in Corollary \ref{curvein2zip} and for $n=3$ up to a factor $(\log B)^\kappa$ in Proposition \ref{curvein3zip} and Corollary \ref{planein3zip}. We also establish a sublinear Schwartz-Zippel bound up to a factor $(\log B)^\kappa$ for curves of arbitrary codimension and for surfaces and 3-folds in low-dimensional spaces. However, in these last bounds the exponent of $d$ is still larger than in Conjecture \ref{SLSZ}. 

    The established sublinear Schwartz-Zippel bounds are used to derive subquadratic dimension growth bounds in different settings.
    
\subsection{Dimension Growth}

    Uniform point counting on curves in \cite{Bombieri-Pila} has led to a broad range of variants and applications. One such direction is Heath-Brown's uniform dimension growth conjecture\footnote{For a broader account of the context and state of the art of the dimension growth conjecture, see \cite{BPW-survey} and \cite[Chapter 2]{Browning-Q}.} \cite{Heath-Brown-Ann}, \cite{Brow-Heath-Salb}, which asks whether for any irreducible variety $X \subset \PP^n$ of dimension $m$ and degree $ d \ge 2$ and any $\varepsilon>0$, one has $$\#X(\QQ,B) \ll_{n,d,\varepsilon} B^{m+\varepsilon}.$$ This conjecture was established by Salberger for $d \ge 4$ in \cite{Salberger-dgc} (including a non-uniform variant for $d=3$, as suggested by Serre \cite{Serre-Mordell}). The dependence on the degree was shown to be at most polynomial in \cite{CCDN-dgc}. Recently, in \cite{BinCluKat}, Binyamini, Cluckers and Kato obtained quadratic dependence on the degree for hypersurfaces of any dimension.

    \begin{thm*}[Quadratic projective dimension growth for hypersurfaces, \cite{BinCluKat}]\label{DG}
        Let $n>1$ be given. There are constants $c=c(n)$ and $\kappa = \kappa(n)$ such that for any $d\ge 4$, any irreducible hypersurface $X$ of degree $d$ in $\PP_\QQ^n$ and any $B > 2$ one has $$\#X(\QQ, B) \le cd^2B^{\dim X}(\log B)^\kappa.$$
    \end{thm*}

    We improve upon this quadratic result in two main ways: by generalising it to varieties of arbitrary codimension and by improving the dependence on the degree in several cases. For varieties of arbitrary codimension, we introduce a natural projection argument with pull-backs that retains the dependence on the degree. This allows us to prove quadratic dependence on the degree for varieties of arbitrary codimension. 

    \begin{thm}[Quadratic projective dimension growth in all codimensions]\label{CDDG}
        Let $n>1$ be given. There are constants $c=c(n)$ and $\kappa = \kappa(n)$ such that for any $d\ge 4$, any irreducible variety $X$ of degree $d$ in $\PP_\QQ^n$ and any $B > 2$ one has $$\#X(\QQ, B) \le cd^2B^{\dim X}(\log B)^\kappa.$$
    \end{thm}

    We also obtain a global field analogue of Theorem \ref{CDDG}; see Theorem \ref{CDDGglobal}. The quadratic dependence on $d$ in Theorem \ref{CDDG} improves upon the polynomial dependence obtained in \cite[Theorem 1]{CCDN-dgc}, which was obtained using projection arguments without pull-backs, leading to a significant increase in the exponent of the degree for higher codimensions.
    
    Theorem \ref{CDDG} works towards a question asked by Cluckers and Glazer in \cite{CGlaz}: to find the optimal dependence on the degree in dimension growth bounds for varieties of arbitrary codimension. Although in affine dimension growth quadratic dependence is optimal (see \cite{BinCluKat}, \cite{CCDN-dgc}), it is likely that the quadratic dependence on the degree in projective dimension growth is suboptimal. Cluckers and Glazer proved the following subquadratic lower bounds for worst case dimension growth.

    \begin{prop*}[Lower bounds for worst case dimension growth, {\cite[Proposition 1]{CGlaz}}]\label{LB}
        For any $n \ge 2$ and any $N$, there exist $c= c(n) > 0$, $B>N$, $d> N$, and an integral hypersurface $X \subset \PP_\QQ^n$ of degree $d$ such that $$cd^{2-\frac{2}{n}}B^{\dim X} \le \#X(\QQ, B).$$
    \end{prop*}

    There remains a small gap in the exponent of the degree between the quadratic upper bound from Theorem \ref{CDDG} and this subquadratic lower bound for the worst case. \cite{CGlaz} asks whether the lower or upper bounds can be improved to establish an optimal dependence on the degree. As a first step in this direction, they show in \cite[Proposition 5]{CGlaz} that for irreducible curves $C \subset \PP^2$ one has $$\#C(\QQ, B) \le cd^{\frac{4}{3}}B(\log B)^\kappa.$$

    We expand upon this result by generalising it to curves of arbitrary codimension. To this end, we prove the following stronger statement for nondegenerate curves.

    \begin{thm}[Subquadratic projective dimension growth for nondegenerate curves]\label{nondegeneratecurves}
        Let $C$ be an irreducible curve in $\PP_\QQ^n$ of degree $d>1$ that is not contained in a hyperplane, and let $B > 2$ be given. Then one has $$\#C(\QQ, B) \le c(n)d^{\frac{n+2}{n+1}}B(\log B)^\kappa$$ for some constant $c(n)$ and absolute constant $\kappa$.
    \end{thm}

    By considering degenerate curves in the space they span, Theorem \ref{nondegeneratecurves} leads to the following bounds.

    \begin{cor}[Subquadratic projective dimension growth for curves]\label{CG5n}
        Let $C$ be a curve in $\PP_\QQ^n$ of degree $d>1$ with no linear components for some $n > 1$ and let $B > 2$ be given. Then one has $$\#C(\QQ, B) \le c(n)d^{\frac{4}{3}}B(\log B)^\kappa$$ for some constant $c(n)$ and absolute constant $\kappa$.
    \end{cor}

    Theorem \ref{nondegeneratecurves} and Corollary \ref{CG5n} lead to some additional observations and questions. The bounds for nondegenerate curves are stronger than those for degenerate curves. One may ask whether this improved dependence on the degree for nondegenerate curves is optimal, and whether similar improved bounds can also be obtained for nondegenerate varieties of general dimensions.

    We also continue the work towards closing the gap between the quadratic upper bounds and subquadratic lower bounds for varieties of dimension $\ge 2$. We use sublinear Schwartz-Zippel bounds to get improved dimension growth bounds. In particular, we obtain subquadratic bounds for hypersurfaces in $\PP^3$, $\PP^4$ and $\PP^5$ in Propositions \ref{surfinn}, \ref{3foldsinn} and \ref{4foldsin5} respectively, as well as for surfaces in arbitrary codimension in Proposition \ref{surfinn}.

    The following theorem summarises our improvements to dimension growth bounds.

    \begin{thm}[Subquadratic projective dimension growth]\label{summary}
        Let $n>1$ be given. There are constants $c=c(n)$ and $\kappa = \kappa(n)$ such that for any $d\ge 4$, any irreducible variety $X$ of degree $d$ and dimension $m$ in $\PP_\QQ^n$ and any $B > 2$ one has $$\#X(\QQ, B) \le cd^{2e(n,m)}B^{\dim X}(\log B)^\kappa,$$ where $e(n,m) \le 1$ and, more specifically, $$e(n,m) = \begin{cases}
            \frac{2}{3} &\text{if } m = 1 \\
            \frac{3(n-1)}{3n-1} &\text{if } m = 2 \\
            \frac{18n-26}{15n-5} &\text{if } m = 3 \text{ and } n \le 6 \\
            \frac{209}{210} &\text{if } m = 4 \text{ and } n = 5 \\
        \end{cases}$$ In all other cases $e(n,m)=1$.
    \end{thm}

    The proofs we give generalise to varieties of arbitrary dimension and codimension when assuming Conjecture \ref{SLSZ}. This leads to the following conditional improvement to quadratic dimension growth.

    \begin{thm}[Subquadratic projective dimension growth]\label{SLSZtoSQDG}
        Suppose Conjecture \ref{SLSZ} holds for varieties of dimension $m-1$ in $\PP_\QQ^n$. Then there are constants $c=c(n)$ and $\kappa = \kappa(n)$ such that for any $d\ge 4$, any integral variety $X$ of degree $d$ and dimension $m$ in $\PP_\QQ^n$ and any $B > 2$ one has $$\#X(\QQ, B) \le cd^{2-\frac{2}{n+1}}B^{m}(\log B)^\kappa.$$
    \end{thm}

    Notice that a small gap remains between the exponent of $d$ in the upper bounds of Theorem \ref{SLSZtoSQDG} and in the lower bounds of \cite[Proposition 1]{CGlaz}. For curves in $\PP^n$ with $n \ge 3$, Corollary \ref{CG5n} is stronger than Theorem \ref{SLSZtoSQDG}. For $n=2$, they give the same bound.
    
    Also note that the exponent of the degree in these bounds depends on the ambient dimension, whereas the bounds of Corollary \ref{CG5n} are independent of the ambient dimension. This leads to the question whether the optimal dependence on the degree may also be independent of the ambient dimension for varieties of higher dimension, in contrast with the dependence on $n$ appearing in Theorem \ref{SLSZtoSQDG}.

\subsection{Affine Varieties}

    The argument used to obtain the result for higher codimension of Theorem \ref{CDDG} can also be applied in different contexts. We illustrate this by proving quadratic dependence on $d$ for all codimensions in Pila's bound for integral points on irreducible affine varieties \cite[Theorem A]{Pila-ast-1995}. This improves upon the polynomial dependence on the degree for arbitrary codimension obtained in \cite[Proposition 2.7]{CDHNV-dgc} and quadratic dependence only for hypersurfaces in \cite[Theorem 3]{BinCluKat}.

\begin{thm}[Quadratic variant of Pila's bound for integral points on affine varieties]\label{CDAff}
        Let $n > 1$ be given. There are constants $c = c(n)$ and $\kappa = \kappa(n)$ such that for any $d>0$, any irreducible variety $X$ of degree $d$ in $\AA_\QQ^n$ and any $B > 2$ one has $$\#X(\ZZ, B) \le cd^2B^{\dim X -1 +\frac{1}{d}}(\log B)^\kappa.$$
    \end{thm}

\subsection*{Structure of the paper}
    In Section \ref{sec:codim} we establish quadratic and conditional subquadratic bounds for general codimension, proving Theorems \ref{CDDG}, \ref{SLSZtoSQDG} and \ref{CDAff}. Section \ref{sec:SZ} treats sublinear Schwartz-Zippel and subquadratic dimension growth: we obtain Theorem \ref{SLSZLin} and several cases of Conjecture \ref{SLSZ} and we prove Theorem \ref{nondegeneratecurves}, Corollary \ref{CG5n} and Theorem \ref{summary}. Finally, in Section \ref{sec:global} we generalise some of our main results to global fields. 

\subsection*{Acknowledgements}
    The author thanks KU Leuven Internal Funds C16/23/010 for partial support. Many thanks go out to Raf Cluckers. Without his excellent guidance and suggestions, this paper would not have been possible. The author would also like to thank Matteo Verzobio for interesting discussions on the topics of the paper and Lillian Pierce, Tim Santens and Matteo Verzobio for helpful feedback on an earlier version.

\renewcommand{\thethm}{\thesection.\arabic{thm}}

\section{Results for General Codimension}\label{sec:codim}
    In this section, we use a natural projection argument to generalise the results of the determinant method obtained for hypersurfaces in \cite{BinCluKat} to varieties in arbitrary codimension. These generalisations are used to derive Theorems \ref{CDDG}, \ref{SLSZtoSQDG} and \ref{CDAff}.

    Our method works by first finding a well-behaved projection, then using the determinant method to find auxiliary hypersurfaces for the projection, and finally pulling those hypersurfaces back through the projection. This way, we obtain suitable auxiliary hypersurfaces for the original variety. We will illustrate two ways one may obtain a suitable projection. The first relies on a result by \cite{CCDN-dgc} to find a projection onto a variety birational to $X$. The second is more elementary and relies on finding a coordinate projection that does not lower the degree of $X$ too much.

    We suspect our technique used to obtain these generalisations to higher codimensions may also be applicable elsewhere. For example, to generalise some results of \cite{Mah}, and to sharpen and simplify some results of \cite{Verzobio}.
\subsection{Projective Varieties}

    Let us start by recalling the application of the determinant method from \cite{BinCluKat}. The following statement and proof are a slight adaptation of \cite[Propositions 7 and 9]{BinCluKat} to the homogeneous context. We introduce the variable exponent $E$ in the prerequisites. The proof remains essentially the same.

    \begin{prop}[{\cite{BinCluKat}}]\label{detmeth}
        Let $X\subset \PP_\QQ^n$ be an irreducible hypersurface of degree $d>0$ for some $n>1$ and let $B \ge 3$. There exist constants $N=N(n), c=c(n), \kappa = \kappa(n)$ such that for any $E>0$ if $$1 \ll_{n,E} (\log B)^N <d<B^E$$ there exist no more than $c(\log B)^\kappa$ many homogeneous polynomials $g_i$ of degree $d-1$ whose joint zero locus contains $X(\QQ, B)$. It follows that $X\not\subset V(g_i)$.
    \end{prop}
    \begin{proof}
        Consider the primes $p_i$ between $\log B$ and $M (\log B)^\ell$ for some large $M(n)$ and $\ell(n)$. Let $\Xi_s$ be the subset of $X(\QQ, B)$ consisting of those points whose reduction modulo each of these primes $p_i$ has multiplicity at least $\frac{d}{(\log B)^\alpha}$ for some large $\alpha(n)$. Then any interpolation determinant with monomials of degree $d':= (\log B)^C$ on the points of $\Xi_s$ is divisible by all these primes $p_i$, for some constant $C = C(n)$, by \cite[Proposition 2]{BinCluKat}. By a calculation as for the high multiplicity points in the proof of \cite[Proposition 7]{BinCluKat}, one finds that $\Xi_s$ lies inside the zero locus of a single polynomial of degree $d'$, and one has $d'<d$ since $(\log B)^N <d$.

        Let us now consider the complement of $\Xi_s$ inside $X(\QQ, B)$, and write it as a finite union of subsets $\Xi_P$ for each of our primes $p_i$ and each low multiplicity point $P$ on $X_{p_i}$, the reduction of $X$ modulo $p_i$. Here low multiplicity means $\frac{d}{(\log B)^\alpha}$. Let $\Delta$ be an interpolation determinant with monomials of degree $d-1$ and points in $\Xi_P$. Write $s$ for the number of monomials of degree $d-1$ in $n+1$ variables.

        By \cite[Corollary 2.5]{CCDN-dgc}, we find that $\Delta$ is divisible by $p_i^e$ with $$e \ge \left(\frac{(n-1)!}{\mu_P}\right)^\frac{1}{n-1}\frac{n-1}{n}s^{1+\frac{1}{n-1}}-as$$ for some constant $a=a(n)$ and with $\mu_P$ the multiplicity of $P$ on $X_{p_i}$. Hence $\Delta$ must vanish by a similar computation as for the low multiplicity points in the proof of \cite[Proposition 7]{BinCluKat}. Now we are done by taking the union over all $p_i$ and all $P$ on $X_{p_i}$ of low multiplicity, similarly as for the proof of \cite[Proposition 7]{BinCluKat}.
    \end{proof}

    Using a projection argument, we can prove the following generalisation of Proposition \ref{detmeth}. We will use this in this section to obtain Theorems \ref{CDDG} and \ref{SLSZtoSQDG}, but Proposition \ref{gendetmeth} will also be important to derive certain sublinear Schwartz-Zippel and subquadratic dimension growth results later on.

    \begin{prop}\label{gendetmeth}
        Let $X\subset \PP_\QQ^n$ be an irreducible variety of dimension $m$ and degree $d>0$ for some $n>m \ge 1$ and let $B \ge 3$. For any $E>0$ there exist constants $N=N(n), c=c(n, E), \kappa = \kappa(n)$ such that if $$1 \ll_{n,E} (\log B)^N <d<B^E$$ there exist no more than $c(\log B)^\kappa$ many hypersurfaces $h_i$ of degree $\le d$ who jointly contain $X(\QQ, B)$ and such that $X \not\subset h_i$.
    \end{prop}
    \begin{proof}
        If $X$ is irreducible but not geometrically irreducible, then let $K/\QQ$ be a finite Galois extension such that $X$ splits into absolutely irreducible components over $K$. Then, since these components are Galois-conjugate, the rational points of $X$ lie in each of these components. Take such a component $Y$. Now let $\sigma$ be a field automorphism of $\overline{\QQ}$ and $Y^\sigma$ the corresponding conjugate variety such that $Y \ne Y^\sigma$. Then there exists a hypersurface $h$ of degree at most $d$ such that $Y \not\subset h$ but $Y^\sigma \subset h$. Hence $X \not\subset h$ and $X(\QQ, B) \subset h$, since $Y(\QQ) \subset Y^\sigma(\QQ)$.

        Now assume $X$ is geometrically irreducible. By \cite[Lemma 5.1]{CCDN-dgc}, there is a linear projection $\pi\colon \PP^n \rightarrow \PP^{m+1}$ such that $X' = \pi(X)$ is an irreducible hypersurface of degree $d$ and $\pi(X(\QQ, B)) \subset X'(\QQ, c_nd^{2(n-m-1)^2}B)$. Proposition \ref{detmeth} now yields no more than $$c_n (\log (c_nd^{2(n-m-1)^2}B))^\kappa < c (\log (c_nB^{2E(n-m-1)^2+1}))^\kappa \le c'(n,m, E)(\log B)^\kappa$$ hypersurfaces $h'_i = V(g_i)$ of degree $d-1$ whose joint zero locus contains $ X'(\QQ, cd^{2(n-m-1)^2}B)$ and since $X'$ is an irreducible hypersurface of degree $d$, $X' \not\subset h'_i$.

        Taking $h_i = \pi^{-1}(h'_i)$ yields the result.
    \end{proof}

    Proposition \ref{gendetmeth} can now be used to obtain Theorems \ref{CDDG} and \ref{SLSZtoSQDG}.

    \begin{proof}[Proof of Theorem \ref{CDDG}]
        Let $m = \dim X$. We may assume that $d < B^{\frac{n-m+1}{2}}$ since if $d \ge B^{\frac{n-m+1}{2}}$ we have $$\#X(\QQ, B) \le \#\PP^n(\QQ, B) \ll B^{n+1}\ll d^2B^m.$$ We may further assume $(\log B)^N < d$ since if $(\log B)^N \ge d$ the result follows from \cite[Theorem 1]{CCDN-dgc} for $d\ge5$. For $d=4$, one can combine the projection argument from \cite[Proposition 4.3.2]{CCDN-dgc} with \cite[Theorem 4]{BinCluKat}.

        Now, assume $(\log B)^N < d$. Let $h_i$ be the degree $d-1$ hypersurfaces from Proposition \ref{gendetmeth}. By Bézout's Theorem, and since the $h_i$ intersect $X$ properly, each intersection $X_i$ of $X$ with $h_i$ is a variety of dimension $m-1$ and degree $\le d(d-1)$. Now by the linear projective Schwartz-Zippel bound, we find $$\#X_i(\QQ, B) \ll d(d-1)B^m.$$ The result now follows by taking the union over all $X_i$, since by Proposition \ref{gendetmeth} there are no more than $c(\log B)^\kappa$ many $X_i$.
    \end{proof}

    Theorem \ref{SLSZtoSQDG} follows by the same argument, using Conjecture \ref{SLSZ} instead of the classical Schwartz-Zippel bound.

    The bounds of Theorem \ref{CDDG} and Theorem \ref{SLSZtoSQDG} are eventually tight as the dimension $m$ grows. 
    
    \begin{prop}[Eventual tightness of projective dimension growth in all codimensions]
        For any $\varepsilon >0$ and $N>1$ there exist $m =m(\varepsilon)$, $B > N$, $d > N$ and for every $n > m$ an irreducible variety $X \subset \PP_\QQ^n$ of dimension $m$ and degree $d$ such that $$d^{2-\varepsilon}B^m \le \#X(\QQ,B).$$
    \end{prop}
    \begin{proof}
        By \cite[Theorem 2]{CGlaz}, for any $\varepsilon>0$ and $N > 1$ there exist $m = m(\varepsilon)$, $B > N$, $d > N$ and an irreducible hypersurface $X \subset \PP^{m+1}$ such that $$d^{2-\varepsilon}B^m \le \#X(\QQ,B).$$ It suffices to view $X$ as a subvariety of $\PP^n$ under the natural embedding of $\PP^{m+1}$ into $\PP^n$.
    \end{proof}

\subsection{Affine Varieties}

    \cite[Lemma 5.1]{CCDN-dgc} is a rather heavy result, but the results of this section can also be obtained by a more elementary argument using coordinate projections. We illustrate this by proving Theorem \ref{CDAff}.

    We need the following lemma to guarantee that the degree does not drop too much under a coordinate projection. Define $$\pi_i: \AA^n \to \AA^{n-1}: (x_1,\dots,x_n) \to (x_1, \dots x_{i-1},x_{i+1},\dots,x_n).$$ 

    \begin{lem}\label{projectiondegree}
        Let $X \subset \AA^n_\QQ$ be an $m$-dimensional irreducible variety of degree $d$, with $m \le n-2$. Then there exists a coordinate projection $\pi_i\colon \AA^n \rightarrow \AA^{n-1}$ such that $\pi_i(X)$ is an irreducible variety of degree $d'$ and $\sqrt{d} \le d' \le d$. If furthermore $m=1$, then $\pi_i(X)$ can be chosen to have dimension $1$.
    \end{lem}
    \begin{proof}
        Since coordinate projections retain irreducibility, $\pi_1(X), \dots, \pi_n(X)$ are irreducible of dimension $m-1$ or $m$, let $d_1,\dots,d_n$ denote their respective degrees. Now $X_i :=\pi_i^{-1}(\pi_i(X))$ is irreducible of dimension $m$ or $m+1$ and of degree $d_i$ for each $i$. Suppose all $X_i$ are equal; then by construction they contain $\AA^n$, contradicting that their dimension is at most $m+1$. Hence $X_i\ne X_j$ for some $i,j$ such that $X_i \cap X_j$ contains $X$ and is of dimension $m$ and degree $\le d_id_j$. Hence $d \le \deg(X_i \cap X_j) \le d_id_j$, such that $d_i$ or $d_j$ is at least $\sqrt{d}$. It is trivial that $d_i,d_j \le d$. If $m = 1$ then either no projection lowers the dimension or $d = 1$. In this final case, taking any projection that retains the dimension suffices.
    \end{proof}

    We may now generalise the affine analogue of Proposition \ref{detmeth}, \cite[Proposition 9]{BinCluKat}, to arbitrary codimension.

    \begin{prop}\label{genaffdetmeth}
        Let $X \subset \AA^n_\QQ$ be an irreducible $m$-dimensional variety of degree $d > 0$ for some $n>m\ge1$ and let $B \ge 3$. There exist constants $N = N(n)$, $c= c(n)$ and $\kappa = \kappa(n)$ such that for any $E>0$, if $$1 \ll_{E} (\log B)^N <d<B^E,$$ there exist no more than $c(\log B)^\kappa$ many hypersurfaces $h_i$ of degree $\le d-1$ who jointly contain $X(\ZZ, B)$ and such that $X \not \subset h_i$.
    \end{prop}
    \begin{proof}
        We proceed by induction on the codimension $n-m$ and the dimension $m$ simultaneously. For $n-m=1$, the result holds by slightly adapting \cite[Proposition 7, Proposition 9]{BinCluKat} similarly to Proposition \ref{detmeth}. Now assume the result holds for codimension $< k$ and for dimension $<m$. Let $n-m=k$. Then by Lemma \ref{projectiondegree} there is a coordinate projection $\pi$ such that $\pi(X)$ is an irreducible variety of degree $d'$ with $\sqrt{d}\le d' \le d$ and either dimension $m-1$ or codimension $k-1$ hence by the induction hypothesis there are no more than $c(\log B)^\kappa$ many hypersurfaces $h_i'$ of degree $\le d-1$ who jointly contain $\pi(X)(\ZZ, B)$ and such that $\pi(X) \not \subset h_i'$. Now the result follows by letting $h_i := \pi^{-1}(h_i')$, since then also $X \not\subset h_i$.
    \end{proof}

    Note that one may take $N(n) = 2^{n-2}N$ with $N$ from \cite[Proposition 9]{BinCluKat}. We are now ready to prove Theorem \ref{CDAff}. The proof is based on that of \cite[Theorem 3]{BinCluKat}.

    \begin{proof}[Proof of Theorem \ref{CDAff}]
        Let $m = \dim X$. We may assume that $d < B^\frac{n-m+1}{2}$ since if $d \ge B^\frac{n-m+1}{2}$ we have $$\#X(\ZZ, B) \le \#\AA^n(\ZZ, B) \le cB^n \le cd^2B^{m-1}.$$ We may further assume $(\log B)^N<d$ since otherwise the result follows from \cite[Proposition 2.7]{CDHNV-dgc}. Now let $h_i$ be the hypersurfaces from Proposition \ref{genaffdetmeth}. By Bézout's Theorem, and since the $h_i$ intersect $X$ properly, each intersection $X_i$ of $X$ with $h_i$ is a variety of dimension $m-1$ and degree $\le d^2$. Now by the affine Schwartz-Zippel bound, we find $$\#X_i(\ZZ, B) \le cd^2B^{m-1}.$$ The result now follows by taking the union over all $X_i$, since by Proposition \ref{genaffdetmeth} there are no more than $c(\log B)^\kappa$ many $X_i$.
    \end{proof}

    Notice that both projection arguments lead to an increased number of hypersurfaces for higher codimension. In the birational approach, this is due to the need to deal with the dependency on the degree in \cite[Lemma 5.1]{CCDN-dgc}. In the approach using coordinate projections, this is due to a higher value of $N$ needed to deal with a potential lowering of the degree under coordinate projections. However, this last approach can possibly be improved to yield a lower value of $\kappa$ by proving a stronger version of Lemma \ref{projectiondegree} with stronger bounds for $d'$ when the codimension of $X$ is larger. It may also be possible to strengthen the dependence on the degree in \cite[Lemma 5.1]{CCDN-dgc} by employing the idea behind sublinear projective Schwartz-Zippel.

\section{Sublinear Schwartz-Zippel and Subquadratic Dimension Growth}\label{sec:SZ}
    The main results of this section are Theorem \ref{SLSZLin} and proofs of Conjecture \ref{SLSZ} and subquadratic dimension growth for varieties of low dimension. In the first subsection, we use results from the geometry of numbers to establish Theorem \ref{SLSZLin} and we show that the dependence on the degree in its bounds is necessarily optimal. In the remainder of the section, we show several cases of sublinear Schwartz-Zippel and subquadratic dimension growth. We start by proving Theorem \ref{nondegeneratecurves} and Corollary \ref{CG5n} and subsequently work towards Theorem \ref{summary} by alternating between proving sublinear Schwartz-Zippel for $m$-dimensional varieties and subquadratic dimension growth for $m+1$-dimensional varieties. The reason for taking turns is that our dimension growth results follow from Schwartz-Zippel bounds for varieties of lower dimensions, whereas some of the proofs of these Schwartz-Zippel bounds rely on dimension growth bounds for varieties of the same dimension.

\subsection{Unions of Linear Varieties}
    We start by proving two results about unions of linear varieties that together provide a strong indication towards Conjecture \ref{SLSZ}. To this end, note the following correspondence between linear varieties in $\PP_\QQ^n$ and sublattices of $\ZZ^{n+1}$.

    By considering the integer points of the affine cone over an $m$-dimensional linear variety $L$ in $\PP_\QQ^n$, we can identify the rational points $L(\QQ)$ of $L$ with (the primitive points of) a primitive sublattice $\Lambda$ of $\ZZ^{n+1}$ of rank $m+1$. It follows that $\#L(\QQ,B) \le \#\Lambda(B)$. Here $\Lambda(B)$ denotes the points of $\Lambda$ with integer coordinates $(x_0,\dots,x_n)$ for which $|x_0|,\dots,|x_n|\le B$. Conversely, we can identify a rank $m+1$ primitive sublattice of $\ZZ^{n+1}$ with an $m$-dimensional linear variety in $\PP^n$ by taking the variety whose affine cone is spanned by the lattice. This results in a bijective correspondence between $m$-dimensional linear varieties in $\PP_\QQ^n$ and rank $m+1$ primitive sublattices of $\ZZ^{n+1}$ in such a way that rational points correspond to primitive lattice points of the same height. Define the height of a linear variety in $\PP_\QQ^n$ or a subspace of $\QQ^{n+1}$ to be the determinant of the corresponding integer lattice. 

    Using this correspondence, some elementary results from the geometry of numbers, recorded by Heath-Brown in \cite{Heath-Brown-Ann} and \cite{Heath-Brown-squarefree}, and a result on the number of rational points of bounded height on Grassmannians by Schmidt \cite{Schmidt-lattices}, one can establish the following result.

    \begin{prop}\label{planecover}
        Let $0\le m \le n-1$ be given. Then there is a constant $c = c(n)$, depending only on $n$, such that there exist $cB^{\frac{(n+1)(n-m)}{n}}$ $m$-dimensional linear subvarieties of $\PP_\QQ^n$ that jointly contain $\PP^n(\QQ,B)$.
    \end{prop}
    \begin{proof}
        By the above correspondence, it suffices to show that there are $cB^{\frac{(n+1)(n-m)}{n}}$ rank $m+1$ primitive lattices jointly containing all primitive $x \in \ZZ^{n+1}(B)$. Notice that for any primitive $x \in \ZZ^{n+1}(B)$ we have $|x|\ll B$, where $|x|$ is the Euclidean norm of $x$. For any primitive $x \in \ZZ^{n+1}(B)$ consider the lattice $$\Lambda_x:= \{y\in\ZZ^{n+1}|x\cdot y = 0\}.$$ By \cite[Lemma 1 (i)]{Heath-Brown-Ann}, $\Lambda_x$ has rank $n$ and determinant $|x|\ll B$. By \cite[Lemma 1 (iii)]{Heath-Brown-Ann}, $\Lambda_x$ admits a basis $b_1,\dots,b_n$ such that $|b_1|\leq |b_2|\leq \dots\leq|b_n|$ and $\prod\limits_{i=1}^n |b_i| \ll|x|$. It follows that $$\prod\limits_{i=1}^{n-m}|b_i|\ll B^\frac{n-m}{n}.$$ Let $\Lambda'_x$ be the primitive lattice generated by $b_1, \dots, b_{n-m}$. Then $\det(\Lambda'_x)\ll B^\frac{n-m}{n}$.
        Now $${\Lambda'}_x^\bot:= \{y\in \ZZ^{n+1}|b_i\cdot y=0, i\in \{1,\dots,n-m\}\}$$ is a rank $m+1$ primitive lattice containing $x$ and by \cite[Lemma 1]{Heath-Brown-squarefree}, $\det({\Lambda'}_x^\bot)=\det(\Lambda'_x)$. Hence any $x\in \ZZ^{n+1}(B)$ lies in a primitive rank $m+1$ lattice of determinant $\ll B^\frac{n-m}{n}$. By \cite[Theorem 1]{Schmidt-lattices}, there are $\le cB^{\frac{(n+1)(n-m)}{n}}$ such lattices for some constant $c = c(n)$ depending only on $n$.
    \end{proof}

    This result is useful in two ways. It shows that the dependence on the degree in Conjecture \ref{SLSZ} is optimal, and it can be used to intersect a variety with a set of linear varieties without losing rational points.

    \begin{remark}\label{rem:optimal}
        Proposition \ref{planecover} shows that the exponent of $d$ in Theorem \ref{SLSZLin} is optimal. If this exponent were any lower, then applying Theorem \ref{SLSZLin} to the union of linear varieties given by Proposition \ref{planecover} would yield a bound for $\#\PP^n(\QQ, B)$ with a lower exponent of $B$ than $n+1$, contradicting that $\#\PP^n(\QQ, B) \gg B^{n+1}$. 

        Similarly, Theorem \ref{SLSZLin} shows that the exponent of $B$ in Proposition \ref{planecover} is optimal, as a lower exponent would lead to a similar contradiction. It also follows that should Conjecture \ref{SLSZ} hold, the exponent of $d$ in its statement is optimal.
    \end{remark}

    Next, we prove Theorem \ref{SLSZLin}. This result is the main inspiration for Conjecture \ref{SLSZ}.

    \begin{proof}[Proof of Theorem \ref{SLSZLin}]

        We may assume that $d \ll B^{n+1}$, since otherwise $$\#X(\QQ,B)\le\#\PP^n(\QQ,B)\ll B^{n+1} \ll d^\frac{n}{n+1}B\le d^\frac{n}{n+1}B^{m+1}$$ and the result holds. We will count the number of rational points of bounded height on the $d$ $m$-planes containing the most rational points. Hence, it suffices to count on the $d$ rank $m+1$ primitive lattices in $\ZZ^{n+1}$ that contain the most points. These are the $d$ lattices with the lowest determinant. To see this, note that by \cite[Theorem 1]{Schmidt-lattices}, there are more than $d$ lattices with determinant  $\ll d^{\frac{1}{n+1}} \ll B$ and that by \cite[Lemma 1 (v)]{Heath-Brown-Ann}, for lattices of determinant $\ll B$, those with the lowest determinant contain the most points. Hence it suffices to count on the lattices $\Lambda$ with $\det \Lambda \le H$, where we set $H := cd^{\frac{1}{n+1}}$. We have $H \ll B$ such that we are in the regime of \cite[Lemma 1 (v)]{Heath-Brown-Ann}.

        Let $\delta= \delta(m) = \frac{1}{m+1}$. By \cite[Theorem 1]{Schmidt-lattices}, there are \begin{align*}
            P(n+1,m+1, M)&-P(n+1, m+1, M-M^{1-\delta}) \\&= a(M^{n+1}-(M-M^{1-\delta})^{n+1})+O(M^{n+1-\delta})\\ &= O(M^{n+1-\delta})
        \end{align*} primitive lattices $\Lambda$ of rank $m+1$ in $\ZZ^{n+1}$ with $$M-M^{1-\delta}< \det(\Lambda)\le M.$$

        Now subdivide the interval $[2, H]$ into $K \le \frac{1}{\delta}H^\delta$ intervals $I_i=[b_{i+1},b_i]$ where $b_{0}=H$ and $b_{K-1}\geq 2$ but $b_{K}<2$ and such that $b_{i+1}=b_{i}-b_{i}^{1-\delta}$. To see this is possible, consider the strictly decreasing function $$f(i) = \delta^{1/\delta}\left(\frac{1}{\delta}H^\delta-i\right)^{1/\delta}.$$ Notice that $f'(i) = -f(i)^{1-\delta}$. By the mean value theorem, $f(i)-f(i+1) = -f'(c) = f(c)^{1-\delta}$ for some $c \in (i, i+1)$, so since $\delta \le 1$ and $f$ is strictly decreasing we have $$f(i)-f(i+1) \le f(i)^{1-\delta}.$$ 
        
        Since $f(0) = H$ and $f(\frac{1}{\delta}H^\delta)=0$, this function divides $[0, H]$ into $\frac{1}{\delta}H^\delta$ intervals $[f(i+1), f(i)]$. Each of these intervals is shorter than the intervals of the required form covering the same part; hence, $[2, H]$ can indeed be covered by no more than $\frac{1}{\delta}H^\delta$ intervals of the required form.

        Now we compute an upper bound for the number $N(B, H)$ of integer points of height $\le B$ that lie on at least one rank $m+1$ primitive lattice of determinant $\le H$, using that by \cite[Lemma 1 (v)]{Heath-Brown-Ann}, a rank $m+1$ lattice of determinant $M$ contains no more than $c\frac{B^{m+1}}{M}$ points of height $\le B$. Let $c_1$ denote the number of primitive rank $m+1$ lattices of determinant $<2$.

        \begin{align*}
            N(B,H) &\ll \sum\limits_{\{\Lambda |\det(\Lambda)\le H\}} \frac{B^{m+1}}{\det(\Lambda)} \\
            &\le c_1B^{m+1}+\sum\limits_{i=0}^{K-1}\sum\limits_{\{\Lambda |\det(\Lambda)\in I_i\}} \frac{B^{m+1}}{\det(\Lambda)} \\
            &\ll B^{m+1}\sum\limits_{i=0}^{K-1}\sum\limits_{\{\Lambda |\det(\Lambda)\in I_i\}} \frac{1}{b_{i+1}} \\
            &\ll B^{m+1}\sum\limits_{i=0}^{K-1} \frac{b_i^{n+{1-\delta}}}{b_{i+1}} \\
            &\ll B^{m+1}\sum\limits_{i=0}^{K-1} b_i^{n-\delta} \\
            &\le B^{m+1}KH^{n-\delta} \le  H^nB^{m+1} \ll d^\frac{n}{n+1}B^{m+1}
        \end{align*}

        Here we used that $$\frac{b_{i+1}}{b_i} = 1-b_{i}^{-\delta}\ge 1-b_{K-1}^{-\delta} \ge 1-2^{-\delta}$$ such that $\frac{b_{i}}{b_{i+1}}$ is bounded above by a constant depending only on $m$.
    \end{proof}

    Theorem \ref{SLSZLin} shows that the classical linear projective Schwartz-Zippel bound cannot be reached by a union of $d$ lines, contrary to the affine Schwartz-Zippel bound. This suggests that linear dependence on $d$ is suboptimal and that the actual sharpest bound of this form is that of Conjecture \ref{SLSZ}.

    In the remainder of this section, we work towards Conjecture \ref{SLSZ} by proving sublinear dependence on $d$ for specific values of $m$ and $n$. Most of these proofs make use of Proposition \ref{planecover} and Theorem \ref{SLSZLin}.

\subsection{Points and Curves}
    We start by proving sublinear Schwartz-Zippel for configurations of points and subquadratic dimension growth for curves. Sublinear Schwartz-Zippel for points follows by a simple argument. Subquadratic dimension growth for curves is obtained as a corollary to \cite[Theorem 1]{BinCluKat} in \cite[Proposition 5]{CGlaz}. There, however, only the case of curves in $\PP^2$ is proven. We will record a full proof here.

    For completeness, we include a separate proof of Conjecture \ref{SLSZ} for $0$-dimensional varieties.

    \begin{prop}\label{pointsinnzip}
        A $0$-dimensional projective variety $X$ of degree $d$ in $\PP_\QQ^n$ contains no more than $cd^\frac{n}{n+1}B$ points of height $\le B$ for some constant $c$ depending only on $n$.
    \end{prop}
    \begin{proof}
        $X$ consists of at most $d$ points, so if $d^\frac{1}{n+1}\le B$ the result follows. If $d^\frac{1}{n+1}\ge B$ then the result follows from the trivial bound $\#\PP^n(\QQ, B) \le cB^{n+1}$.
    \end{proof}

    Notice that Proposition \ref{pointsinnzip} could also be obtained as a special case of Theorem \ref{SLSZLin}.
    
    We continue by proving Theorem \ref{nondegeneratecurves} and Corollary \ref{CG5n} to establish subquadratic dimension growth for curves of any codimension. We start by considering nondegenerate curves. 
    
    \begin{proof}[Proof of Theorem \ref{nondegeneratecurves}]
        If $\log B \ge d$, then the result follows from \cite[Theorem 1]{BinCluKat}. Hence, we may suppose $\log B < d$. By \cite[Theorem 1]{BinCluKat}, we obtain $$\#C(\QQ, B) \le c_1(n)d^2(\log B)^\kappa$$ for some constant $c_1(n)$ and absolute constant $\kappa$. If we now assume that $d^\frac{n}{n+1}<B$, the result follows.

        The remaining case, when $d^\frac{n}{n+1} \ge B$, is treated as follows: By Proposition \ref{planecover}, $\PP^n(\QQ, B)$ can be covered by $\ll B^\frac{n+1}{n}$ hyperplanes. Since the intersection of $C$ with each hyperplane consists of at most $d$ points we obtain $$\#C(\QQ, B) \ll dB^\frac{n+1}{n}.$$ Combining this with $d^\frac{n}{n+1} \ge B$ yields the result.
    \end{proof}
    Corollary \ref{CG5n} now follows by applying Theorem \ref{nondegeneratecurves} to $C$ as an irreducible curve in the lowest-dimensional plane containing $C$.
    
\subsection{Curves and Surfaces}
    For curves in $\PP^2$, Conjecture \ref{SLSZ} can be shown using only Proposition \ref{planecover} and Theorem \ref{SLSZLin}.

    \begin{cor}[Sublinear projective variant of Schwartz-Zippel for curves in $\PP^2$]\label{curvein2zip}
        Let $C$ be a projective curve of degree $d$ in $\PP_\QQ^2$ and let $B \ge 1$ be given. Then one has $$\#C(\QQ, B) \le cd^{\frac{2}{3}}B^2$$ for some absolute constant $c$.
    \end{cor}
    \begin{proof}
        By Theorem \ref{SLSZLin}, we may assume $C$ has no linear factors. Then the intersection of $C$ with the $cB^\frac{3}{2}$ lines jointly containing $\PP^2(\QQ,B)$ provided by Proposition \ref{planecover} consists of at most $cdB^\frac{3}{2}$ points. Now we are done if $d^\frac{1}{3}\le B^\frac{1}{2}$. If $d^\frac{1}{3}\ge B^\frac{1}{2}$, then we find $$\#C(\QQ, B)\le \#\PP^2(\QQ, B)\le cB^3 \le cd^{\frac{2}{3}}B^2.$$
    \end{proof}

    For curves and surfaces in $\PP^3$, we obtain Conjecture \ref{SLSZ} up to a logarithmic factor.

    \begin{prop}[Sublinear projective variant of Schwartz-Zippel for curves in $\PP^3$]\label{curvein3zip}
        Let $C$ be a projective curve of degree $d$ in $\PP_\QQ^3$ and let $B > 2$ be given. Then one has $$\#C(\QQ, B) \le cd^{\frac{3}{4}}B^2 (\log B)^\kappa$$ for some absolute constants $c$ and $\kappa$.
    \end{prop}

    We will establish Proposition \ref{curvein3zip} as a special case of Proposition \ref{curveinnzip} below. To establish sublinear Schwartz-Zippel for curves in all dimensions, we distinguish three cases: linear components, which are treated by Theorem \ref{SLSZLin}; non-linear components that lie in planes; and components that do not lie in planes. The second case is the most involved and is treated in the following Lemma.

    \begin{lem}\label{curvesinplanes}
        Let $C$ be a projective curve of degree $d$ in $\PP_\QQ^n$ that does not have any linear components and such that each irreducible component of $C$ is contained in a plane. Let $B > 2$ be given. Then one has $$\#C(\QQ,B) \le c(n)d^\frac{n}{n+1}B^2(\log B)^\kappa$$ for some constants $c(n)$ and $\kappa(n)$.
    \end{lem}
    \begin{proof}
        Note that we may assume that $d \le B^\frac{(n+1)(n-1)}{n}\le B^{n+1}$ since if $d \ge B^\frac{(n+1)(n-1)}{n}$ we have $$\#C(\QQ,B) \le \#\PP^n(\QQ,B)\leq cB^{n+1}\le cd^\frac{n}{n+1}B^2.$$

        Write $C$ as the union of roughly $\log d \ll \log B$ curves $C_i$, where $C_i$ consists of all components of $C$ with degree between $d_i$ and $2d_i$ for some $d_i \le d$. It suffices to show the result for each $C_i$ separately and then sum the results.

        Suppose $d_i \le \log B$, then, by Theorem \ref{CG5n}, each component of $C_i$ contains no more than $cB(\log B)^\kappa$ points of bounded height for some constants $c= c(n)$ and $\kappa=\kappa(n)$ depending only on $n$. Since $C_i$ has at most $d$ components we obtain $$\#C_i(\QQ,B) \ll dB(\log B)^\kappa \le d^\frac{n}{n+1}B^2(\log B)^\kappa,$$ such that the result holds for $C_i$.

        Now suppose $\log B \le d_i \le d^\frac{1}{n-1}$, then by \cite[Theorem 1]{BinCluKat}, each component of $C_i$ contains no more than $cd_i^2(\log B)^\kappa$ points of bounded height for some constants $c=c(n)$ and $\kappa= \kappa(n)$ depending only on $n$. Since $C_i$ has at most $\frac{d}{d_i}$ components we obtain $$\#C_i(\QQ,B) \ll dd_i(\log B)^\kappa \le d^\frac{n}{n-1}(\log B)^\kappa \le d^\frac{n}{n+1}B^2(\log B)^\kappa,$$ where we used that $d \le B^\frac{(n+1)(n-1)}{n}$. Thus, we obtain the result for $C_i$.

        Next, suppose $d^\frac{1}{n-1} \le d_i$, then again by \cite[Theorem 1]{BinCluKat}, each component of $C_i$ contains no more than $cd_i^2(\log B)^\kappa$ points of bounded height for some constants $c= c(n)$ and $\kappa = \kappa(n)$ depending only on $n$ and since $C_i$ has at most $\frac{d}{d_i}$ components we have $\#C_i(\QQ,B) \ll dd_i(\log B)^\kappa$. Now suppose we also have $d_i^{\frac{2n+1}{3(n+1)}}d^{\frac{1}{3(n+1)}} \le B$, then we recover $$\#C_i(\QQ,B)\ll d_i^{-\frac{n-1}{3(n+1)}}d^\frac{3n+1}{3(n+1)}B^2(\log B)^\kappa \le d^\frac{n}{n+1}B^2(\log B)^\kappa,$$ and the result holds for $C_i$. 

        In the remaining case, where $d^\frac{1}{n-1} \le d_i$ and $d_i^{\frac{2n+1}{3(n+1)}}d^{\frac{1}{3(n+1)}} \ge B$, note that each component is contained in a plane, such that $C_i$ is contained in a union of at most $\frac{d}{d_i}$ planes. Hence, by Theorem \ref{SLSZLin}, we have $$\#C_1(\QQ,B) \ll \left(\frac{d}{d_i}\right)^\frac{n}{n+1}B^3 \le d_i^{-\frac{n-1}{3(n+1)}}d^\frac{3n+1}{3(n+1)}B^2\le d^{\frac{n}{n+1}}B^2.$$

        Now summing the results over all $C_i$ completes the proof.
    \end{proof}

    Combining Lemma \ref{curvesinplanes}, Proposition \ref{planecover} and Theorem \ref{SLSZLin} leads to sublinear Schwartz-Zippel bounds for curves.

    \begin{prop}[Sublinear projective variant of Schwartz-Zippel for curves in $\PP^n$]\label{curveinnzip}
        Let $C$ be a projective curve of degree $d$ in $\PP_\QQ^n$ for some $n\ge 3$ and let $B > 2$ be given. Then one has $$\#C(\QQ, B) \le cd^{\frac{3(n-1)}{3n-1}}B^2 (\log B)^\kappa$$ for some constants $c(n)$ and $\kappa(n)$.
    \end{prop}
    \begin{proof}
        Write $C = C_1 \cup \dots \cup C_{n}$ where $C_k$ consists of all components that lie in a $k$-plane, but not in a $k-1$-plane. Hence, $C_1$ consists of all linear components, $C_2$ of all non-linear components that lie in a plane and so on. It suffices to prove the proposition for each of these curves separately and sum the results.

        For $C_1$, the result follows from Theorem \ref{SLSZLin}. For $C_2$, the result follows from Lemma \ref{curvesinplanes}. For $3 \le k \le n$, we proceed as follows. By Proposition \ref{planecover}, all points of height $\le B$ in a $k$-plane can be covered by $cB^\frac{k+1}{k}$ $k-1$-planes. Intersecting each component of $C_k$ with $cB^\frac{k+1}{k}$ $k-1$-planes yields $$\#C_k(\QQ,B)\leq cdB^\frac{k+1}{k} \le cdB^\frac{4}{3}.$$ If now $d^\frac{2}{3n-1} \leq B^\frac{2}{3}$ then the result follows. If $d^\frac{2}{3n-1} \ge B^\frac{2}{3}$, then we obtain $$\#C(\QQ, B) \le \#\PP^n(\QQ,B) \le cB^{n+1} \le cd^{\frac{3(n-1)}{3n-1}}B^2.$$
    \end{proof}

    For $n = 3$ we recover Proposition \ref{curvein3zip}. For $n\ge4$, there is still room for improvement in the exponent of $d$. Note that without Lemma \ref{curvesinplanes}, a similar proof would lead to the exponent $\frac{2(n-1)}{2n-1}$ in Proposition \ref{curveinnzip}. The logarithmic factor in this result comes from the application of \cite[Theorem 1]{BinCluKat} and the need to distinguish between $\log d$ different regimes of degrees.

    We move on to dimension growth for surfaces. Combining Propositions \ref{curveinnzip} and \ref{gendetmeth} yields the following result.

    \begin{prop}[Subquadratic dimension growth for surfaces in $\PP^n$]\label{surfinn}
        Let $X$ be an irreducible projective surface of degree $d>3$ in $\PP_\QQ^n$ and let $B > 2$ be given. Then one has $$\#X(\QQ,B) \le cd^{\frac{6(n-1)}{3n-1}}B^2 (\log B)^\kappa$$ for some constants $c(n)$ and $\kappa(n)$.
    \end{prop}

    \begin{proof}
        We may assume that $d < B^\frac{3n-1}{6}$ since if $d \ge B^{\frac{3n-1}{6}}$ we have $$\#X(\QQ, B) \le \#\PP^n(\QQ, B) \le cB^{n+1}\le cd^\frac{6(n-1)}{3n-1}B^2.$$ We may further assume $(\log B)^N<d$ with $N=N(n)$ as in Proposition \ref{gendetmeth} since if $(\log B)^N \ge d$ the result follows from Theorem \ref{CDDG}. Now let $h_i$ be the hypersurfaces from Proposition \ref{gendetmeth}. By Bézout's Theorem, each intersection $X_i$ of $X$ with $h_i$ is a curve of degree $\le d(d-1)$. Now by Proposition \ref{curveinnzip} we find $$\#X_i(\QQ, B) \le c(d(d-1))^\frac{3(n-1)}{3n-1}B^2(\log B)^\kappa \le cd^\frac{6(n-1)}{3n-1}B^2(\log B)^\kappa.$$ Now we are done by summing over all $X_i$.
    \end{proof}

    Note that the exponent of $d$ in the bound of Proposition \ref{CG5n} is independent of $n$, whereas the exponent in the bound of Proposition \ref{surfinn} grows with $n$, even when assuming Conjecture \ref{SLSZ}. This indicates that the dimension growth bounds obtained here and in Theorem \ref{SLSZtoSQDG} are likely suboptimal for higher codimensions.

\subsection{Surfaces and Threefolds}
    We continue moving up to higher dimensions by using Propositions \ref{gendetmeth}, \ref{surfinn} and \ref{planecover} and Theorem \ref{SLSZLin} to prove sublinear Schwartz-Zippel bounds for surfaces.

    Up to a logarithmic factor, Conjecture \ref{SLSZ} for surfaces in $\PP^3$ can be obtained as a corollary to Proposition \ref{curvein3zip}.

    \begin{cor}[Sublinear projective variant of Schwartz-Zippel for surfaces in $\PP^3$]\label{planein3zip}
        Let $X$ be a projective surface of degree $d$ in $\PP_\QQ^3$ and let $B > 2$ be given. Then one has $$\#X(\QQ,B) \le cd^\frac{3}{4}B^3(\log B)^\kappa$$ for some absolute constants $c$ and $\kappa$.
    \end{cor}
    \begin{proof}
        By Theorem \ref{SLSZLin}, we may assume $X$ has no linear factors. Then by intersecting $X$ with the $cB^\frac{4}{3}$ planes jointly containing $\PP^3(\QQ,B)$ provided by Proposition \ref{planecover} we obtain a curve $C$ of degree at most $cdB^\frac{4}{3}$ such that $X(\QQ, B) \subset C(\QQ, B)$. Now Proposition \ref{curvein3zip} yields the result.
    \end{proof}

    We also obtain a suboptimal but still sublinear dependence for surfaces in $\PP^4$, $\PP^5$ and $\PP^6$.

    \begin{prop}[Sublinear projective variant of Schwartz-Zippel for surfaces in $\PP^4$,$\PP^5$ and $\PP^6$]\label{surfaceinnzip}
        Let $X$ be a projective surface of degree $d$ in $\PP_\QQ^n$ for some $n \in \{4,5,6\}$ and let $B > 2$ be given. Then one has $$\#X(\QQ, B) \le cd^{\frac{18n-26}{15n-5}}B^3 (\log B)^\kappa$$ for some constants $c(n)$ and $\kappa(n)$.
    \end{prop}

    Similarly as for curves, we first discuss the case of components lying inside 3-planes separately. Without this, we would end up with the exponent $\frac{15n-21}{12n-4}$, which would not yield a sublinear result for $n=6$ and less sharp results for $n=4,5$.

    \begin{lem}\label{surfacesin3planes}
        Let $n \ge 4$. Let $X$ be a projective surface of degree $d$ in $\PP_\QQ^n$ that does not have any linear, quadratic or cubic components and such that each irreducible component is contained in a 3-plane. Let $B > 2$ be given. Then one has $$\#X(\QQ, B) \le cd^{\frac{3(n-1)}{3n-1}}B^3 (\log B)^\kappa$$ for some constants $c(n)$ and $\kappa(n)$.
    \end{lem}

    \begin{proof}
        First notice that we may assume $d \le B\frac{(n-2)(3n-1)}{3(n-1)}\le B^\frac{3n-1}{2}$ since otherwise we have $$\#X(\QQ, B) \le \#\PP^n(\QQ, B) \le cB^{n+1} \le cd^{\frac{3(n-1)}{3n-1}}B^3.$$ Now write $X$ as the union of roughly $\log d \ll \log B$ surfaces $X_i$, such that $X_i$ contains all components of degree between $d_i$ and $2d_i$ for some $\log B \le d_i \le d$. It suffices to establish the result for each $X_i$ separately and sum the results. 

        Now suppose $d_i \le (\log B)^N$ with $N = N(n)$ from Proposition \ref{gendetmeth}. Then by Proposition \ref{surfinn}, each component of $X_i$ contains no more than $cB^2(\log B)^\kappa$ points of bounded height for some constants $c$ and $\kappa$, depending only on $n$. So since $X_i$ has no more than $d$ components we obtain $$\#X_i(\QQ,B) \ll dB^2(\log B)^\kappa \le d^{\frac{3(n-1)}{3n-1}}B^3(\log B)^\kappa,$$ where we used that $d \le B^\frac{3n-1}{2}$. Hence, the result holds for $X_i$.

        Next, suppose $(\log B)^N \le d_i$, then for each irreducible component $Y_j$ of $X_i$, Proposition \ref{gendetmeth} yields $\le c(\log B)^\kappa$ hypersurfaces of degree at most $2d_i$ such that their joint intersection with $Y_j$ is a curve containing $Y_j(\QQ, B)$ of degree $\ll d_i^2(\log B)^\kappa$. Now take the union of these curves for each $Y_j$. Since there are at most $\frac{d}{d_i}$ components $Y_j$, this union is a curve of degree $\ll dd_i(\log B)^\kappa$ containing $X(\QQ, B)$. Applying Proposition \ref{curveinnzip} now yields $$\#X_i(\QQ,B) \ll (dd_i)^\frac{3(n-1)}{3n-1}B^2(\log B)^\kappa.$$ If now $ B \ge d^\frac{n-3}{2(3n-1)(n+1)}d_i^\frac{6n^2-n-3}{2(3n-1)(n+1)}$, we obtain $$\#X_i(\QQ, B) \ll d^\frac{6n^2-n-3}{2(3n-1)(n+1)}d_i^\frac{n-3}{2(3n-1)(n+1)}B^3(\log B)^\kappa \le d^\frac{3(n-1)}{3n-1}B^3(\log B)^\kappa,$$ where we used that $d_i \le d$. 

        In the remaining case, where $ B\le d^\frac{n-3}{2(3n-1)(n+1)}d_i^\frac{6n^2-n-3}{2(3n-1)(n+1)}$, note that $X_i$ is contained in at most $\frac{d}{d_i}$ 3-planes. By Theorem \ref{SLSZLin} these planes jointly contain no more than $c\left(\frac{d}{d_i}\right)^\frac{n}{n+1}B^4$ points. This yields $$\#X_i(\QQ, B) \ll \left(\frac{d}{d_i}\right)^\frac{n}{n+1}B^4 \le d^\frac{6n^2-n-3}{2(3n-1)(n+1)}d_i^\frac{n-3}{2(3n-1)(n+1)}B^3 \le d^\frac{3(n-1)}{3n-1}B^3.$$

        Now summing the obtained results for each $X_i$ yields $$\#X(\QQ,B) \ll d^\frac{3(n-1)}{3n-1}B^3(\log B)^{\kappa+1} ,$$ proving the result.
    \end{proof}

    \begin{proof}[Proof of Proposition \ref{surfaceinnzip}]
        Write $X = X_q \cup X_c \cup X_2 \cup \dots \cup X_{n}$ where $X_q$ consists of all quadratic components, $X_c$ of all cubic components, and $X_k$ consists of all components of degree $\ne 2,3$ that lie in a $k$-plane, but not in a $k-1$-plane. Hence, $X_2$ consists of all linear components. It suffices to establish the bounds for each of these surfaces separately and sum the results.

        For each component $x$ of $X_q$ we have after projecting onto a 3-plane that by \cite[Theorem 2]{Heath-Brown-Ann}, $\#x(\QQ, B) \le c_\epsilon B^{2+\epsilon}$ for any $\epsilon >0$. By setting $\epsilon = \frac{2}{n}$ and summing over all factors we obtain $\#X_q(\QQ, B) \le c_\frac{2}{n}dB^{2+\frac{2}{n}}$. If now $d^\frac{1}{n+1}\le B^\frac{n-2}{n}$ then $\#X_q(\QQ, B) \ll d^\frac{n}{n+1}B^3$ and if $d^\frac{1}{n+1}\ge B^\frac{n-2}{n}$ then $\#X_q(\QQ, B) \le \#\PP^n(\QQ,B) \ll B^{n+1} \le d^\frac{n}{n+1}B^3$, such that the bound holds for $X_q$.

        For $X_c$ a similar argument can be employed, now relying on \cite[Theorem 0.4]{Salberger-dgc} with $\epsilon = \frac{2}{n}+1-\frac{2}{\sqrt{3}}$.

        For $X_2$, the result follows from Theorem \ref{SLSZLin}. For $X_3$, the result follows from Lemma \ref{surfacesin3planes}. For $4 \le k \le n$, we proceed as follows. By Proposition \ref{planecover} all rational points of height $\le B$ in a $k$-plane can be covered by $cB^\frac{k+1}{k}$ $k-1$-planes. Intersecting each component of $X_k$ with $cB^\frac{k+1}{k}$ $k-1$-planes yields a curve of degree at most $cdB^\frac{k+1}{k}$ and applying Proposition \ref{curveinnzip} to this curve yields $$\#X_k(\QQ,B)\ll d^\frac{3(n-1)}{3n-1}B^{2+\frac{k+1}{k}\frac{3(n-1)}{3n-1}}(\log B)^\kappa \le d^\frac{3(n-1)}{3n-1}B^{2+\frac{5}{4}\frac{3(n-1)}{3n-1}}(\log B)^\kappa.$$ If now $d^\frac{4}{5} \ge B$ then the result follows. If $d^\frac{4}{5} \le B$ then we apply Proposition \ref{surfinn} to each component to obtain $$\#X_k(\QQ, B) \ll d^{2\frac{3(n-1)}{3n-1}}B^2(\log B)^\kappa \le  d^{\frac{18n-26}{15n-5}}B^3(\log B)^\kappa.$$
    \end{proof}

    We can, in turn, use Proposition \ref{surfaceinnzip} to prove subquadratic dimension growth for threefolds.

    \begin{prop}[Subquadratic dimension growth for 3-folds in $\PP^4$,$\PP^5$ and $\PP^6$]\label{3foldsinn}
        Let $X$ be an irreducible projective 3-fold of degree $d>3$ in $\PP_\QQ^n$ for some $n \in \{4,5,6\}$ and let $B > 2$ be given. Then one has $$\#X(\QQ,B) \le cd^{2\frac{18n-26}{15n-5}}B^3 (\log B)^\kappa$$ for some constants $c(n)$ and $\kappa(n)$.
    \end{prop}

    \begin{proof}
        We may assume that $d < B^\frac{(n-2)(n+1)}{2n}$ since if $d \ge B^\frac{(n-2)(n+1)}{2n}$ we have $$\#X(\QQ, B) \le \#\PP^n(\QQ, B) \ll B^{n+1}\le d^{\frac{2n}{n+1}}B^3.$$ We may further assume $(\log B)^N<d$ with $N=N(n)$ as in Proposition \ref{gendetmeth} since if $(\log B)^N \ge d$ the result follows from Theorem \ref{CDDG}. Now let $h_i$ be the hypersurfaces from Proposition \ref{gendetmeth}. By Bézout's Theorem, each intersection $X_i$ of $X$ with $h_i$ is a surface of degree $\le d(d-1)$. Now by Proposition \ref{surfaceinnzip} we find $$\#X_i(\QQ, B) \ll (d(d-1))^\frac{18n-26}{15n-5}B^3(\log B)^\kappa \le d^{2\frac{18n-26}{15n-5}}B^3(\log B)^\kappa.$$ Now we are done by taking the union over all $X_i$.
    \end{proof}
\subsection{Threefolds and Fourfolds}

    Sublinear Schwartz-Zippel for hypersurfaces in $\PP^4$ follows readily from Proposition \ref{surfaceinnzip}, similarly as for hypersurfaces in $\PP^2$ and $\PP^3$.

    \begin{cor}[Sublinear projective variant of Schwartz-Zippel for 3-folds in $\PP^4$]\label{3foldsin4zip}
        Let $X$ be a projective 3-fold of degree $d$ in $\PP_\QQ^4$ and let $B > 2$ be given. Then one has $$\#X(\QQ,B) \le cd^\frac{48}{55}B^4(\log B)^\kappa$$ for some absolute constants $c$ and $\kappa$.
    \end{cor}
    \begin{proof}
        By Theorem \ref{SLSZLin}, we may assume $X$ has no linear factors. Quadratic and cubic factors can be dealt with as in the proof of Proposition \ref{curveinnzip}. Then by intersecting $X$ with the $cB^\frac{5}{4}$ planes jointly containing $\PP^4(\QQ,B)$ provided by Proposition \ref{planecover} we obtain a surface $X'$ of degree at most $cdB^\frac{5}{4}$ such that $X(\QQ, B) \subset X'(\QQ, B)$. Now Proposition \ref{surfaceinnzip} applied to $X'$ yields $$\#X(\QQ, B) \ll d^{\frac{46}{55}}B^{3+\frac{23}{22}} (\log B)^\kappa.$$ Combining this with the bound obtained by applying Proposition \ref{3foldsinn} to each component of degree $\ge 4$ yields the result.
    \end{proof}
    
    The final case of Sublinear Schwartz-Zippel we can obtain using our methods is that of threefolds in $\PP^5$.

    \begin{prop}[Sublinear projective variant of Schwartz-Zippel for 3-folds in $\PP^5$]\label{3foldsin5zip}
        Let $X$ be a projective 3-fold of degree $d$ in $\PP_\QQ^5$ and let $B > 2$ be given. Then one has $$\#X(\QQ,B) \le cd^\frac{209}{210}B^4(\log B)^\kappa$$ for some absolute constants $c$ and $\kappa$.
    \end{prop}

    \begin{proof}
        Write $X = X_q \cup X_c \cup X_3 \cup X_4 \cup X_{5}$ where $X_q$ consists of all quadratic components, $X_c$ of all cubic components, and $X_k$ consists of all components of degree $\ne 2,3$ that lie in a $k$-plane, but not in a $k-1$-plane. In particular, $X_3$ consists of all linear components. It suffices to prove the proposition for each of these 3-folds separately and sum the results. $X_q$ and $X_c$ can be treated using \cite[Theorem 2]{Heath-Brown-Ann} with $\epsilon = \frac{3}{5}$ and \cite[Theorem 0.4]{Salberger-dgc} with $\epsilon = \frac{3}{5}+1-\frac{2}{\sqrt{3}}$, as in the proof of Proposition \ref{surfaceinnzip}.

        For $X_3$, the result follows from Theorem \ref{SLSZLin}. For $X_4$, an argument as in the proof of Lemma \ref{surfacesin3planes} yields the bound $\#X_4(\QQ, B) \le cd^\frac{32}{35}B^4(\log B)^\kappa$. For $X_5$, one can argue as in the proof of Corollary \ref{3foldsin4zip} to obtain the bound $\#X_5(\QQ, B) \le cd^\frac{209}{210}B^4(\log B)^\kappa$.
    \end{proof}

    We may now conclude by proving the remaining case of subquadratic dimension growth in $\PP^5$.

    \begin{prop}[Subquadratic dimension growth in $\PP^5$]\label{4foldsin5}
        Let $X$ be an irreducible projective 4-fold of degree $d>3$ in $\PP_\QQ^5$ and let $B > 2$ be given. Then one has $$\#X(\QQ,B) \le cd^{\frac{209}{105}}B^4 (\log B)^\kappa$$ for some absolute constants $c$ and $\kappa$.

        It follows that there exist constants $c, \kappa$ and $\epsilon(m) >0$ such that for any irreducible projective variety $X \subset \PP^5$ of dimension $m$ and any $B > 2$ one has $$\#X(\QQ,B) \le cd^{2-\epsilon}B^{\dim X} (\log B)^\kappa.$$
    \end{prop}

    \begin{proof}
        We may assume that $d < B^\frac{(n-2)(n+1)}{2n}$ since if $d \ge B^\frac{6}{5}$ we have $$\#X(\QQ, B) \le \#\PP^5(\QQ, B) \ll B^{6}\le d^{\frac{5}{3}}B^4.$$ We may further assume $(\log B)^N<d$ with $N=N(n)$ as in Proposition \ref{gendetmeth} since if $(\log B)^N \ge d$ the result follows from Theorem \ref{CDDG}. Now let $h_i$ be the hypersurfaces from Proposition \ref{gendetmeth}. By Bézout's Theorem, each intersection $X_i$ of $X$ with $h_i$ is a 3-fold of degree $\le d(d-1)$. Now by Proposition \ref{3foldsin5zip} we find $$\#X_i(\QQ, B) \ll (d(d-1))^\frac{209}{210}B^4(\log B)^\kappa \le d^{\frac{209}{105}}B^4(\log B)^\kappa.$$ Now we are done by taking the union over all $X_i$. The second assertion follows by combining Propositions \ref{CG5n}, \ref{surfinn} and \ref{3foldsinn}.
    \end{proof}

    Now all cases of Theorem \ref{summary} have been proven in Propositions \ref{CG5n}, \ref{surfinn}, \ref{3foldsinn}, \ref{4foldsin5} and Theorem \ref{CDDG}.

\section{Global fields}\label{sec:global}

We generalise our main results to all global fields instead of just $\QQ$. The principal outcomes are global‑field analogues of Theorem \ref{SLSZLin} and Theorem \ref{CDDG}. With these generalisations established, all of the results in Section \ref{sec:SZ} also generalise to global fields.

Let $K$ be a global field, namely, a finite field extension of $\QQ$ or $\FF_q(t)$ for some prime power $q$. Let $d_K$ denote the degree of extension of $K$ over $\QQ$ or $\FF_q(t)$.
For a subvariety $X$ of $\PP_K^n$ and $B\ge 1$, write $X(K,B)$
for the set of $x$ in $X(K)$ with $H_K(x)\le B$, with the relative height $H_K$ on $\PP_K^n(K)$ as defined just below (4.1) of \cite{CDHNV-dgc}.
The height $H(S)$ of a subspace $S$ of $K^{n+1}$ is defined in \cite{Thunder92} for number fields and in \cite{Thunder2008} for function fields.

\subsection{Higher Codimension}

First we deduce quadratic dimension growth for varieties of any codimension over any global field.

Let $p_i \in \PP_K^{n}$ denote the point $(0:\dots:0:1:0:\dots:0)$ with a $1$ in the $(i+1)\text{-st}$ coordinate. Now define the $i\text{-th}$ coordinate projection $\pi_i: \PP^n_K\setminus\{p_i\}\to \PP^{n-1}_K: (x_0:\dots:x_n)\to (x_0:\dots:x_{i-1}:x_{i+1}:\dots:x_n).$

\begin{lem}\label{projectiondegree_projective}
        Let $X \subset \PP^n_K$ be an $m$-dimensional irreducible variety of degree $d$, with $m \le n-2$. Then there exists a coordinate projection $\pi_i\colon \PP^n_K\setminus\{p_i\} \rightarrow \PP^{n-1}_K$ for some $i \in \{0,\dots,n\}$ such that $\overline{\pi_i(X\setminus\{p_i\})}$ is an irreducible variety of degree $d'$ and $\sqrt{d} \le d' \le d$. If furthermore $m=1$, then $\overline{\pi_i(X\setminus\{p_i\})}$ can be chosen to have dimension $1$.
    \end{lem}
    \begin{proof}
        Since coordinate projections retain irreducibility, $\overline{\pi_1(X\setminus\{p_1\})}, \dots, \overline{\pi_n(X\setminus\{p_n\})}$ are irreducible of dimension $m-1$ or $m$, let $d_1,\dots,d_n$ denote their respective degrees. Now $X_i :=\overline{\pi_i^{-1}(\pi_i(X\setminus\{p_i\}))}$ is irreducible of dimension $m$ or $m+1$ and of degree $d_i$ for each $i$. Suppose all $X_i$ are equal; then by construction they must contain $\PP^n$, contradicting that their dimension is at most $m+1$. Hence $X_i\ne X_j$ for some $i,j$ such that $X_i \cap X_j$ contains $X$ and is of dimension $m$ and degree $\le d_id_j$. Hence $d \le \deg(X_i \cap X_j) \le d_id_j$, such that $d_i$ or $d_j$ is at least $\sqrt{d}$. It is trivial that $d_i,d_j \le d$. If $m = 1$ then either no projection lowers the dimension or $d = 1$. In this final case, taking a coordinate projection that retains the dimension suffices.
    \end{proof}

The following proposition generalises Proposition \ref{gendetmeth}.

    \begin{prop}\label{gendetmethglobal}
        Let $X\subset \PP_K^n$ be an irreducible variety of dimension $m$ and degree $d>0$ for some $n>m \ge 1$ and let $B \ge 3$. For any $E>0$ there exist constants $N=N(n), c=c(K, n, E), \kappa = \kappa(n)$ such that if $$1 \ll_{K,n,E} (\log B)^N <d<B^E$$ there exist no more than $c(\log B)^\kappa$ many hypersurfaces $h_i$ of degree $\le d-1$ who jointly contain $X(K, B)$ and such that $X \not\subset h_i$.
    \end{prop}
    \begin{proof}
        We proceed by induction on the codimension $n-m$ and dimension $m$ of $X$. The base case for $n-m=1$ can be obtained by slightly adapting \cite[Proposition 13]{BinCluKat} as in Proposition \ref{detmeth}.

        Let $n-m=k$ and assume the result holds for codimension $< k$ and for dimension $< m$. By Lemma \ref{projectiondegree_projective} there is a coordinate projection $\pi_i$ such that $\overline{\pi(X\setminus\{p_i\})}$ is an irreducible variety of degree $d'$ with $\sqrt{d}\le d' \le d$ and either dimension $m-1$ or codimension $k-1$ hence by the induction hypothesis there are no more than $c(\log B)^\kappa$ many hypersurfaces $h_i'$ of degree $\le d-1$ who jointly contain $\overline{\pi_i(X\setminus\{p_i\})}(K, B)$ and such that $\overline{\pi_i(X\setminus\{p_i\})} \not \subset h_i'$. Now the result follows by letting $h_i := \overline{\pi^{-1}(h_i')}$, since then also $X \not\subset h_i$.
    \end{proof}

    Proposition \ref{gendetmethglobal} can now be used to obtain a global field analogue of Theorem \ref{CDDG}. 

    \begin{thm}[Quadratic projective dimension growth in all codimensions]\label{CDDGglobal}
        Let $n>1$ be given. There are constants $c=c(K, n)$ and $\kappa = \kappa(n)$ such that for any $d\ge 4$, any irreducible variety $X$ of degree $d$ in $\PP_K^n$ and any $B > 2$ one has $$\#X(K, B) \le cd^2B^{\dim X}(\log B)^\kappa.$$
    \end{thm}

    \begin{proof}
        Let $m = \dim X$. We may assume that $d < B^{\frac{n-m+1}{2}}$ since if $d \ge B^{\frac{n-m+1}{2}}$ we have $$\#X(K, B) \le \#\PP_K^n(K, B) \ll B^{n+1}\ll d^2B^m.$$ We may further assume $(\log B)^N < d$ since if $(\log B)^N \ge d$ the result follows from \cite[Theorem 4.19]{CDHNV-dgc}.

        Now, assume $(\log B)^N < d$. Let $h_i$ be the degree $d-1$ hypersurfaces from Proposition \ref{gendetmethglobal}. By Bézout's Theorem, and since the $h_i$ intersect $X$ properly, each intersection $X_i$ of $X$ with $h_i$ is a variety of dimension $m-1$ and degree $\le d(d-1)$. Now by the Schwartz-Zippel bound we find $$\#X_i(K, B) \ll d(d-1)B^m.$$ The result now follows by taking the union over all $X_i$, since by Proposition \ref{gendetmethglobal} there are no more than $c(\log B)^\kappa$ many $X_i$.
    \end{proof}

    Theorem \ref{CDDGglobal} generalises \cite[Theorem 4]{BinCluKat} to varieties of any codimension for $d \ge 4$. 

\subsection{Unions of linear varieties}

The second goal of this section is to extend Theorem \ref{SLSZLin} to arbitrary global fields.

\begin{thm}[Sublinear projective variant of Schwartz-Zippel for a union of linear varieties]\label{SLSZLin:global}
    Let $d\ge 1$ and $0\le k \le n-1$ be given.
    Let $X$ be the union of $d$ many $m$-dimensional linear subvarieties of $\PP_K^n$. Then there exists a constant $c=c(K, n)$ depending only on $n$ and $K$ such that
    $$
    \#X(K,B) \le cd^{\frac{n}{n+1}}B^{m+1}.
    $$
\end{thm}
\begin{proof}
    The proof of Theorem \ref{SLSZLin} can be adapted to work for all number fields. Here, \cite[Theorem 1]{Thunder93linear} replaces \cite[Lemma 1 (v)]{Heath-Brown-Ann}, \cite[Theorem 1]{Schmidt-lattices} can be replaced by \cite[Theorem 1]{Thunder92} and \cite[Corollary to Theorem 3]{Schanuel} can be applied to count the number of points of bounded height in $\PP_K^n$. The remainder of the proof can then be repeated, now using $\delta(m) = \frac{1}{d_K(m+1)}$. 

    For function fields, the proof can be simplified when compared to that of Theorem \ref{SLSZLin}, because heights of subspaces only take discrete powers of $q$ as values. We work with the additive height $h(L)$ of a linear variety $L$, which takes only positive integer values. Recall that $H(L)=q^{h(L)}$. It suffices to count the number of $K$-rational points of relative height $\le B$ on the $d$ $m+1$-dimensional linear subspaces of lowest height in $K^{n+1}$. By \cite[Corollary 4]{Thunder2008}, it suffices to count on the subspaces $L\subset K^{n+1}$ with $h(L)\le h$, where $h :=\frac{1}{n+1}\log_q d + a $ and $a$ is a constant depending only on $n$ and $K$. 

    We first deal with the case $q^h \gg B$. Then $d \gg B^{n+1}$ such that using \cite[Theorem 4.1]{Wan} we obtain $$\#X(K,B)\le\#\PP_K^n(K,B)\ll B^{n+1} \ll d^\frac{n}{n+1}B\le d^\frac{n}{n+1}B^{m+1}$$ and the result follows. Hence, we may assume $q^h\ll B$ such that we are in the regime of the second part of \cite[Theorem 1]{Thunder2008}.
    
    Let $N(B, h)$ denote the number of rational points of height $\le B$ that lie on an $m+1$-dimensional subspace $L$ of additive height $h(L)\le h$. Now, relying on \cite[Theorem 1]{Thunder2008} to count the number of points on a subspace of given height and on \cite[Corollary 4]{Thunder2008} to count the number of subspaces $S(m+1, i)$ of a given additive height $i$ and dimension $m+1$, we obtain the following bound: 

    \begin{align*}
        N(B, h) &\le \sum\limits_{\{L | h(L) \le h\}}\#L(K,B) \ll \sum\limits_{\{L | h(L) \le h\}} \frac{B^{m+1}}{q^{h(L)}} \\ &= B^{m+1}\sum\limits_{i=0}^h\frac{S(m+1,i)}{q^i} \\ &\ll B^{m+1}\sum\limits_{i=0}^h \frac{q^{(n+1)i}}{q^i} = \frac{q^{n(h+1)}-1}{q^n-1}B^{m+1}\\ &\ll q^{nh}B^{m+1} \ll d^\frac{n}{n+1}B^{m+1}.
    \end{align*}

    This completes the proof since $\#X(K, B) \le N(B, h)$.
    \end{proof}

    Theorem \ref{SLSZLin:global} improves the classical projective Schwartz-Zippel bound over global fields, which follows from the proof of Theorem 1 of \cite{BrowningHeathBrown-Crelle}. It also generalises the special case of $X = \PP_K^n$ given by
    $$
    \# \PP^n_K(K,B)\le c_{n,K} B^{n+1},
    $$
    from \cite{Schanuel}, \cite[p.~19]{Serre-Mordell}, \cite{Wan}. 

\begin{lem}[Classical projective variant of Schwartz-Zippel     bounds]\label{lem:proj:SZ:K}
    Let $X$ be a (possibly reducible) subvariety of $\PP_K^n$ of pure dimension $m$ and degree $d\ge 1$.
    Then one has
    $$
    \# X(K,B) \le cdB^{m+1}
    $$
    for some $c=c(K, n)$ depending only on $n$ and $K$.
\end{lem}

To see that the dependence on the degree in Theorem \ref{SLSZLin:global} is optimal, note the following generalisation of Proposition \ref{planecover}.

\begin{prop}\label{planecover:global}
    Let $0\le m<n$ be given. Then there is $c = c(K, n)$ such that there exist $cB^{\frac{(n+1)(n-m)}{n}}$ $m$-dimensional linear subvarieties of $\PP_K^n$ that jointly contain $\PP^n(K,B)$.
\end{prop}
\begin{proof}
    The proof of Proposition \ref{planecover} can be adapted to all global fields. Instead of the discriminant of the corresponding integer lattice, we work with the height $H$ of a subspace of $K^{n+1}$ as defined in \cite{Thunder92} and \cite{Thunder2008}. The duality result of \cite[Lemma 1]{Heath-Brown-squarefree} is replaced by the duality theorems from \cite{Thunder93flag} for number fields and from \cite{Thunder2008} for function fields, with $A$ the identity matrix. Using this duality, it follows readily that the subspace defined by $S_x:= \{y\in K^{n+1}|x\cdot y=0\}$ has dimension $n$ and height $H(S_x)=H(x)$. This replaces \cite[Lemma 1 (i)]{Heath-Brown-Ann}. \cite[Lemma 1 (iii)]{Heath-Brown-Ann} is replaced by \cite[Theorem 9]{Bombi-Vaal} for number fields and by \cite[Corollary 2]{Thunder95Siegel} for function fields (see also \cite[Theorem 2.3]{Pared-Sas}). Lastly, \cite[Theorem 1]{Schmidt-lattices} is replaced by \cite[Theorem 1]{Thunder92} for number fields and \cite[Corollary 4]{Thunder2008} for function fields. 
\end{proof}

Now, relying on these generalisations, most of our other main results adapt more easily to general $K$, as the results of \cite{BinCluKat} and \cite{CDHNV-dgc} (which generalise and refine \cite{CCDN-dgc} and \cite{Pared-Sas}) hold over any global field $K$. For example, our Theorem \ref{nondegeneratecurves}, Proposition \ref{curvein3zip} and Corollary \ref{planein3zip} also hold over global fields when considering the relative height.

\bibliographystyle{amsalpha}
\bibliography{bookshelf}

\end{document}